\documentclass[a4paper,10pt,reqno, english]{amsart}  
\usepackage[utf8]{inputenc}
\usepackage[T1]{fontenc}
\usepackage{amsmath,amsthm}
\usepackage{amsfonts,amssymb,enumerate}
\usepackage{url,paralist}
\usepackage{mathtools}  
\usepackage[colorlinks=true,urlcolor=blue,linkcolor=red,citecolor=magenta]{hyperref}
\usepackage{enumerate}
\usepackage{anysize}

\theoremstyle{plain}
\newtheorem{theorem}{Theorem}[section]
\newtheorem{lemma}[theorem]{Lemma}

\newtheorem{corollary}[theorem]{Corollary}

\newtheorem{conjecture}[theorem]{Conjecture}

\newtheorem*{theorem*}{Theorem}

\newtheorem*{claim*}{Claim}

\theoremstyle{definition}

\newtheorem{remark}[theorem]{Remark}

\newcommand{\lf}{\left\lfloor}
\newcommand{\rf}{\right\rfloor}
\newcommand{\lc}{\left\lceil}
\newcommand{\rc}{\right\rceil}
\newcommand{\F}{\mathcal{F}}
\renewcommand{\P}{\mathcal{P}}
\newcommand{\s}{\textbf{s}}
\newcommand{\R}{\mathbb{R}}
\newcommand{\Z}{\mathbb{Z}}

\newcommand{\conv}{\mathrm{conv}}
\newcommand{\KG}{\mathrm{KG}}

\renewcommand{\F}{\mathcal{F}}

\newcommand{\tcd}{\mathrm{tcd}}

\begin{document}

\title[On the generalized Erd\H os--Kneser conjecture]{On the generalized Erd\H os--Kneser conjecture:\\ proofs and reductions}



\author{Jai Aslam} 
\address[JA]{Dept.\ Math., Northeastern University, Boston, MA}
\email{aslam.j@husky.neu.edu} 

\author{Shuli Chen}
\address[SC]{Dept.\ Math., Cornell University, Ithaca, NY}
\email{sc2586@cornell.edu} 

\author{Ethan Coldren}
\address[EC]{Dept.\ Math., Colorado State University, Fort Collins, CO}
\email{ethancorvids@gmail.com} 

\author{Florian Frick}
\address[FF]{Mathematical Sciences Research Institute, Berkeley, CA}
\address{Dept.\ Math., Cornell University, Ithaca, NY}
\email{ff238@cornell.edu} 

\author{Linus Setiabrata}
\address[LS]{Dept.\ Math., Cornell University, Ithaca, NY}
\email{ls823@cornell.edu} 

\date{\today}
\maketitle


\begin{abstract}
\small
Alon, Frankl, and Lov\'asz proved a conjecture of Erd\H os that one needs at least $\lceil \frac{n-r(k-1)}{r-1} \rceil$ 
colors to color the $k$-subsets of $\{1, \dots, n\}$ such that any $r$ of the $k$-subsets that have the same color 
are not pairwise disjoint. A generalization of this problem where one requires $s$-wise instead of pairwise intersections was
considered by Sarkaria. He claimed a proof of a generalized Erd\H os--Kneser conjecture establishing
a lower bound for the number of colors that reduces to Erd\H os' original conjecture for~${s = 2}$.
Lange and Ziegler pointed out that his proof fails whenever $r$ is not a prime.
Here we establish this generalized Erd\H os--Kneser conjecture for every~$r$, as long as $s$ is not too
close to~$r$. Our result encompasses earlier results but is significantly more general.
We discuss relations of our results to conjectures of Ziegler and of Abyazi Sani and Alishahi, and prove
the latter in several cases.
\end{abstract}

\section{Introduction}

Lov\'asz~\cite{lovasz1978} proved a conjecture of Kneser~\cite{kneser1955} that at least $n-2(k-1)$ colors are 
required in any coloring of the $k$-element subsets of $[n] = \{1, \dots, n\}$ to ensure that any two $k$-subsets 
of the same color intersect. We will denote the $k$-element subsets of a set $X$ by~$\binom{X}{k}$. 
Erd\H os~\cite{erdos1973} conjectured that at least $\lceil \frac{n-r(k-1)}{r-1} \rceil$
colors are required in any coloring of $\binom{[n]}{k}$ to ensure that any $r$ elements of $\binom{[n]}{k}$ that 
receive the same color are not pairwise disjoint. This was proven by Alon, Frankl, and Lov\'asz~\cite{alon1986}. 

Sarkaria~\cite{sarkaria1990} considered the natural extension, where one seeks to understand this combinatorial 
question beyond the case of pairwise intersections: He showed that for $s\ge 2$ at least $\lceil \frac{(s-1)n-r(k-1)}{r-1} \rceil$ 
colors are required in any coloring of $\binom{[n]}{k}$ to ensure that given any $r$ (not necessarily distinct) elements 
of $\binom{[n]}{k}$ that receive the same color some $s$ of them intersect, provided that $r$ is a prime. While 
Sarkaria claimed this result for arbitrary~$r$, his reduction to the prime case contained an error and only works 
if $s$ is at most the largest prime divisor of~$r$; see Lange and Ziegler~\cite{lange2007} and Ziegler~\cite{ziegler2006}. 
Thus for larger $s$ and composite $r$ this generalized Erd\H os--Kneser conjecture has remained open.

Here we prove the generalized Erd\H os--Kneser conjecture for arbitrary~$r$, and $s$ not too close to~$r$ (see
Theorem~\ref{thm:main2} where we prove a more general statement):

\begin{theorem}
\label{thm:main}
	Let $r\ge 2$, $k \ge 1$, $n \ge sk$ be integers. Let $r = 2^{\alpha_0}\cdot p_1^{\alpha_1} \cdots p_t^{\alpha_t}$
	be the prime factorization of $r$, i.e., the $p_i$ are pairwise distinct odd primes and $\alpha_i \in \Z_{\ge 0}$. Let 
	$2 \le s \le 2^{\alpha_0}\cdot (p_1-1)^{\alpha_1} \cdots (p_t-1)^{\alpha_t}$.
	Then at least $\lceil \frac{(s-1)n-r(k-1)}{r-1} \rceil$ colors are required in any coloring
	of $\binom{[n]}{k}$ to ensure that given any $r$ (not necessarily distinct) elements of $\binom{[n]}{k}$ that 
	receive the same color some $s$ of them intersect.
\end{theorem}

Throughout the manuscript for an integer $r \ge 2$ with prime factorization $2^{\alpha_0}\cdot p_1^{\alpha_1} \cdots p_t^{\alpha_t}$
we define $b(r) = 2^{\alpha_0}\cdot (p_1-1)^{\alpha_1} \cdots (p_t-1)^{\alpha_t}$.
We emphasize that we distinguish ``intersect'' and ``pairwise intersect'', where $s$ sets intersect by definition if they all share
a common element.
This lower bound on the number of colors is known to be not optimal; see Lange and Ziegler~\cite{lange2007}. In fact, they conjectured
that their upper bound of $1+\lc \frac{1}{\lc \frac{r-1}{s-1} \rc}\frac{(s-1)n-rk+1}{s-1} \rc$ is tight.

Our proof does not rely on Sarkaria's result, but instead employs techniques recently developed by the fourth author~\cite{frick2017, frick2017-2}.
In particular, our approach also yields new proofs of earlier results. We relate the generalized Erd\H os--Kneser conjecture
to several other problems: Among others we observe that it follows from extending the optimal colored Tverberg theorem of
Blagojevi\'c, Matschke, and Ziegler~\cite{blagojevic2009} from primes $r$ to prime powers; see Remark~\ref{rem}.

We present some preliminaries on hypergraphs such as the standard translation of the generalized Erd\H os--Kneser conjecture 
into a hypergraph coloring problem in Section~\ref{sec:transition}. Section~\ref{sec:proof} contains the proof of our main
result and relations to other problems. We discuss a conjecture of Abyazi Sani and Alishahi in Section~\ref{sec:conj}
and prove it in several cases. In fact, we prove a strengthening of the conjecture in those cases in Section~\ref{sec:subhyp}; 
see Theorem~\ref{thm:t-wide}. We use this to establish that one can restrict to a natural subfamily of all $k$-element 
subsets of~$[n]$ and still observe the same intersection pattern as in Theorem~\ref{thm:main}; see Corollary~\ref{cor:subhyp}.

\section{Transition to hypergraph colorings}
\label{sec:transition}

Given a ground set $X$ and an integer $r\ge 2$ an \emph{$r$-uniform hypergraph} (with multiplicities) $H$ 
is a set of $r$-element multisubsets of~$X$, that is, repetitions of elements are allowed.
Elements of $X$ are the \emph{vertices} of $H$, while elements of~$H$ are called \emph{hyperedges}.
An \emph{$m$-coloring} of $H$ for an integer~${m \ge 1}$ is a map $c\colon X \longrightarrow [m]$ such that 
every hyperedge~${\sigma \in H}$ contains vertices $x,y \in \sigma$ with different colors ${c(x) \ne c(y)}$. The smallest
integer $m$ such that $H$ admits an $m$-coloring is the \emph{chromatic number} of~$H$, denoted by~$\chi(H)$.

There is a standard reformulation of Kneser's conjecture and its relatives in the language of hypergraph colorings. 
The vertex set is~$\binom{[n]}{k}$. If $\KG^r(n,k)$ is the hypergraph whose hyperedges are collections of $r$ pairwise
disjoint $k$-subsets in~$[n]$, then the result of Alon, Frankl, and Lov\'asz states that $\chi(\KG^r(n,k)) \ge \lceil \frac{n-r(k-1)}{r-1} \rceil$.
A simple greedy coloring shows that this lower bound is optimal.

Following Lange and Ziegler~\cite{lange2007}, 
denote by $\KG^r_{s-1}(n,k)$ the hypergraph on vertex set~$\binom{[n]}{k}$, where a multisubset $\{\{A_1, \dots, A_r\}\}$
of $\binom{[n]}{k}$ forms a hyperedge of $\KG^r_{s-1}(n,k)$ if and only if $A_{i_1} \cap \dots \cap A_{i_s} = \emptyset$ 
for any pairwise distinct $i_1, \dots, i_s \in [r]$. In Theorem~\ref{thm:main} we seek to establish a lower bound for 
the chromatic number of~$\KG^r_{s-1}(n,k)$. Indeed, if $\chi(\KG^r_{s-1}(n,k)) \ge \lceil \frac{(s-1)n-r(k-1)}{r-1} \rceil$ 
then in any coloring with less than $\lceil \frac{(s-1)n-r(k-1)}{r-1} \rceil$ colors there are $r$ (not necessarily distinct)
$k$-element sets $A_1, \dots, A_r$ that have the same color and form a hyperedge of~$\KG^r_{s-1}(n,k)$, that is,
any $s$ of them (with multiplicities) have empty intersection.
More generally, for $\s = (s_1, \dots, s_n)$ with $s_i \in [r-1]$, Lange and Ziegler define the hypergraph $\KG^r_\s(n,k)$ 
on vertex set~$\binom{[n]}{k}$, where a multisubset $\{\{A_1, \dots, A_r\}\}$ of $\binom{[n]}{k}$ forms a hyperedge of 
$\KG^r_\s(n,k)$ if and only if each $i \in [n]$ is contained in at most $s_i$ of the sets $A_1, \dots, A_r$. 
In particular, $\KG^r_{s-1}(n,k) = \KG^r_{(s-1, \dots, s-1)}(n,k)$.

We will establish lower bounds for $\chi(\KG^r_{s-1}(n,k))$ 
by comparing it to the chromatic number of a Kneser hypergraph whose vertex set consists of transversals
of a partition of~$[n]$: Let $\P = \{P_1, \dots, P_\ell\}$ be a partition of $[n]$; denote by $\KG^r(n,k; \P)$ 
(or by $\KG^r(n,k; P_1, \dots, P_\ell)$) the hypergraph whose vertices are $k$-element subsets 
$\sigma \subset [n]$ with $|\sigma \cap P_i| \le 1$ for all $i \in [\ell]$, and whose hyperedges consists
of $r$ such sets that are pairwise disjoint. More generally, if $\F$ is any system of subsets of~$[n]$,
let $\KG^r(\F; \P)$ be the hypergraph whose vertices correspond to those $\sigma \in \F$ with $|\sigma\cap P_i| \le 1$
for all~${i \in [\ell]}$, and with a hyperedge for any collection of $r$ vertices that correspond to pairwise disjoint
sets.

Lower bounds for the chromatic number of $\KG^r(\F; \P)$ follow from~\cite{frick2017-2} provided that $r$ is a prime and
the partition consists of sets of size at most~${r-1}$:

\begin{lemma}[{\cite[Thm.~4.7]{frick2017-2}}]
\label{lem:tcd}
	Let $\F$ be a system of subsets of $[n]$, let $r$ be a prime, and let $\P = \{P_1, \dots, P_\ell\}$ be a partition
	of $[n]$ with~${|P_i| \le r-1}$ for all~${i \in [\ell]}$. Then 
	$$\chi(\KG^r(\F;\P)) \ge \lc \frac{\tcd^r(\F)}{r-1} \rc.$$
\end{lemma}

We briefly recall the definition of the \emph{topological $r$-colorability defect} $\tcd^r(\F)$ of a system $\F$ of subsets
of~$[n]$ from~\cite{frick2017-2}. Let $K$ be a simplicial complex on vertex set $[N]$ with $N \ge n$ whose minimal 
nonfaces are the inclusion-minimal sets of~$\F$. Let $d \ge 0$ be the least integer such that there is a continuous
map $f\colon K \longrightarrow \R^d$ that does not identify $r$ points from $r$ pairwise disjoint faces of~$K$.
Now $\tcd^r(\F)$ is the maximum over all numbers $N-(r-1)(d+1)$ that arise in this way. 
With~\cite[Lemma~4.5]{frick2017-2} we easily see that $\tcd^r(\binom{[n]}{k}) \ge n-r(k-1)$, and thus we immediately 
get the following lemma:

\begin{lemma}
\label{lem:bounds}
	Let $r \ge 2$ be a prime, and let $k \ge 1$ and $n \ge rk$ be integers. Let $\P = \{P_1, \dots, P_\ell\}$ be a partition
	of $[n]$ with~${|P_i| \le r-1}$ for all~${i \in [\ell]}$. Then
	$$\chi(\KG^r(n,k;\P)) = \lc \frac{n-r(k-1)}{r-1} \rc.$$
\end{lemma}

\section{Proof of the main result}
\label{sec:proof}

Sarkaria~\cite{sarkaria1990} claimed to establish a lower bound for $\chi(\KG^r_{s-1}(n,k))$ for arbitrary~$r$,
but his induction on the number of prime divisors of~$r$ contained a mistake that cannot be fixed; 
see Lange and Ziegler~\cite{lange2007}. Here we first show that $\chi(\KG^r_{s-1}(n,k))$ is bounded 
from below by the chromatic number of a specially constructed hypergraph, see Lemma~\ref{lem:transversal-red}, 
and then we can induct on prime divisors of~$r$ for this family of hypergraphs in a non-straightforward
way, see Lemma~\ref{lem:prime-ind}.

\begin{lemma}
\label{lem:transversal-red}
	Let $r \ge 2$, $k \ge 1$, and $n \ge k$ be integers and $\s = (s_1, \dots, s_n)$ a vector of positive integers~${s_i \in [r-1]}$.
	Let $N= \sum_i s_i$ and let $\P = \{P_1, \dots, P_n\}$ be a partition of~$[N]$ with~${|P_i| = s_i}$.
	Then $\chi(\KG^r_{\s}(n,k)) \ge \chi(\KG^r(N, k; \P))$.
\end{lemma}

\begin{proof}
	Given an $m$-coloring $c \colon \binom{[n]}{k} \longrightarrow [m]$ of the hypergraph $\KG^r_{\s}(n,k)$
	we will define an $m$-coloring $c'$ of~$\KG^r(N, k; \P)$. Let $f\colon [N] \longrightarrow [n]$
	be the map that projects every element of $P_i$ to~$i$. For a $k$-element subset $A \subset [N]$
	with $|A \cap P_i| \le 1$ for all~$i$ define $c'(A) = c(f(A))$.
	
	We claim that $c'$ defines an $m$-coloring: Let $A_1, \dots, A_r \subset [N]$ be sets with $c'(A_1)=\dots = c'(A_r)$.
	Since $c$ is an $m$-coloring, the vertices $f(A_1), \dots, f(A_r)$ do not form a hyperedge 
	of~$\KG^r_{\s}(n,k)$. Thus there is an element $\ell \in [n]$ and pairwise distinct 
	$i_1, \dots, i_{s+1}$ with ${\ell \in f(A_{i_1}) \cap \dots \cap f(A_{i_{s+1}})}$, where $s = s_\ell$.
	Then $P_\ell \cap A_{i_j} \ne \emptyset$ for each~${j \in [s+1]}$. 
	Since $P_\ell$ only has $s$ elements, the sets $A_{i_1}, \dots, A_{i_{s+1}}$ cannot be pairwise disjoint.
	Hence the sets $A_1, \dots, A_r$ do not form a hyperedge of~$\KG^r(N, k; \P)$ and $c'$ is an $m$-coloring.
\end{proof}

The following reduction to prime divisors is similar to the one given by Alon, Drewnoswki, and \L uczak~\cite[Lemma~3.3]{alon2009}.

\begin{lemma}
\label{lem:prime-ind}
	Let $r_1, r_2 \ge 2$ and $b_1, b_2 \ge 1$ be integers. Suppose for all integers ${k_1 \ge 1}$, $n_1\ge r_1k_1$, 
	and partitions $\P=\{P_1, \dots, P_l\}$ of $[n_1]$ with $|P_i| \le b_1$ for all~$i$, we have 
	$\chi(\KG^{r_1}(n_1, k_1; \P)) \ge \lc \frac{n_1 - r_1(k_1-1)}{r_1-1} \rc$, 
	and suppose for all integers $k_2 \ge 1$, $n_2\ge r_2k_2$, and partitions $\mathcal Q = \{Q_1, \dots, Q_m\}$ 
	of $[n_2]$ with $|Q_j| \le b_2$ for all~$j$, we have 
	$\chi(\KG^{r_2}(n_2, k_2 ;\mathcal Q)) \ge \lc \frac{n_2 - r_2(k_2-1)}{r_2-1}\rc$. 
	Then for all integers $k \ge 1$, $n\ge r_1r_2k$, and partitions $\mathcal R = \{R_1, \dots, R_t\}$ of $[n]$ 
	with $|R_i| \leq b_1b_2$ for all~$i$, we have
	$\chi(\KG^{r_1r_2}(n,k;\mathcal R)) \ge \lc \frac{n - r_1r_2(k-1)}{r_1r_2-1}\rc$.
\end{lemma}

\begin{proof}
	Assume $r_1 \le r_2$. Suppose there exists an $L$-coloring $c$ of $\KG^{r_1r_2}(n,k; \mathcal R)$ where 
	$L< \lc \frac{n - r_1r_2(k-1)}{r_1r_2-1}\rc$. Note that $(r_1r_2-1)L \le n-r_1r_2(k-1)-1$. For $K_1 = L(r_2-1)+r_2(k-1)+1$ 
	we now have
	$$r_1(K_1-1) + L(r_1-1) = r_1L(r_2-1)+r_1r_2(k-1) + L(r_1-1) = r_1r_2(k-1)+L(r_1r_2-1) \le n-1.$$

	For all $A \subset [n]$ with $|A| = K_1$ such that $|R_i \cap A| \le b_2$, our assumption says that 
	for the partition $\mathcal R' = \{R_1 \cap A, \dots, R_t \cap A\}$ of $A$ we have
	$\chi(\KG^{r_2}(\binom{A}{k};\mathcal R')) \ge \lc \frac{K_1 - r_2(k-1)}{r_2-1}\rc \ge L+1$, so any $L$-coloring 
	of it yields a monochromatic hyperedge. In particular, the $L$-coloring of $\KG^{r_2}(\binom{A}{k};\mathcal R')$ 
	induced by $c$ yields a monochromatic pairwise disjoint $r_2$-tuple of $k$-sets in~$A$.

	Now partition each $R_i$ into $R_{i,1}, \dots, R_{i,\alpha_i}$ with $\alpha_i \le b_2$ and $|R_{i,j}| \le b_1$. 
	Each vertex $A$ of the hypergraph $\KG^{r_1}(n, K_1;R_{1,1}, \dots, R_{1,{\alpha_1}}, \dots, R_{t,1}, \dots, R_{t,{\alpha_t}})$ 
	contains at most $\alpha_i \le b_2$ elements from each~$R_i$. Then by the previous paragraph, each vertex 
	$A$ contains a monochromatic pairwise disjoint $r_2$-tuple, so we can give an $L$-coloring $c'$ of this hypergraph 
	by assigning each vertex $A$ the color of one of the monochromatic pairwise disjoint $r_2$-tuples it contains. 
	Then by assumption we have 
	$$\chi(\KG^{r_1}(n, K_1;R_{1,1}, \dots, R_{1,{\alpha_1}}, \dots, R_{t,1}, \dots, R_{t,{\alpha_t}})) \ge \lc \frac{n - r_1(K_1-1)}{r_1-1}\rc \ge L+1,$$ 
	so this $L$-coloring of $\KG^{r_1}(n,K_1;R_{1,1}, \dots, R_{1,{\alpha_1}}, \dots, R_{t,1}, \dots, R_{t,{\alpha_t}})$ 
	yields a monochromatic hyperedge $(A_1, \dots, A_{r_1})$. Since for each $i \in [r_1]$ the set $A_i$ contains a monochromatic 
	pairwise disjoint $r_2$-tuple of $k$-sets of the same color, together we have $r_1r_2$ pairwise disjoint $k$-sets of the 
	same color in the original hypergraph $\KG^{r_1r_2}(n,k;R_1, \dots, R_t)$. 
	This contradicts that $c$ is a proper coloring.
\end{proof}

As a last ingredient before proving our main result we need one additional reduction for $r$ a power of two. 
For an integer $s \ge 1$ a set $\sigma \subset [n]$ is called \emph{$s$-stable} if any two of its elements are at 
distance at least~$s$ in the cyclic ordering of~$[n]$. Denote by $\KG^r(n,k)_{s-\textrm{stab}}$ the subhypergraph
of~$\KG^r(n,k)$ induced by those vertices that correspond to $s$-stable sets. Ziegler conjectured:

\begin{conjecture}[Ziegler~\cite{ziegler2002}]
\label{conj:stab}
	For integers $r \ge 2$, $k \ge 1$, and $n \ge rk$ we have
	$$\chi(\KG^r(n,k)_{r-\textrm{stab}}) = \lc \frac{n-r(k-1)}{r-1} \rc.$$
\end{conjecture}

The $r=2$ case of this conjecture recovers a classical result of Schrijver~\cite{schrijver1978}. 
Alon, Drewnowski, and \L uczak~\cite{alon2009} showed that if Conjecture~\ref{conj:stab} holds for
$r = p$ and for $r = q$ then it also holds for their product~$pq$. Thus in light of Schrijver's result
Conjecture~\ref{conj:stab} holds for $r$ a power of two. Otherwise it is open; for some recent progress
see~\cite{frick2017-2}. 

If $\P = \{P_1, \dots, P_\ell\}$ is a partition of~$[n]$ with $|P_i| \le r$ for all~${i \in [\ell]}$, then
the hypergraph $\KG^r(n,k)_{r-\textrm{stab}}$ is a subhypergraph of $\KG^r(n, k; \P)$ after possibly
reordering~$[n]$ to make each part $P_i$ consist of consecutive elements. Together
with the preceding paragraph this implies:

\begin{lemma}
\label{lem:2^t}
	Let $r = 2^t$ be a power of two, and let $k \ge 1$ and $n \ge rk$ be integers.
	Let $\P = \{P_1, \dots, P_\ell\}$ be a partition of $[n]$ with~${|P_i| \le r}$ for all~${i \in [\ell]}$. Then
	$$\chi(\KG^r(n, k; \P)) = \lc \frac{n-r(k-1)}{r-1} \rc.$$
\end{lemma}

We conjecture that $r$ being a power of two is a superfluous condition in the preceding lemma:

\begin{conjecture}
\label{conj:part}
	Let $r \ge 2$, $k \ge 1$, and $n \ge rk$ be integers.
	Let $\P = \{P_1, \dots, P_\ell\}$ be a partition of $[n]$ with~${|P_i| \le r}$ for all~${i \in [\ell]}$. Then
	$$\chi(\KG^r(n, k; \P)) = \lc \frac{n-r(k-1)}{r-1} \rc.$$
\end{conjecture}

Conjecture~\ref{conj:part} is implied by Conjecture~\ref{conj:stab}. Lemma~\ref{lem:prime-ind} shows
that it suffices to establish Conjecture~\ref{conj:part} for $r$ a prime. 
Combining the lemmas above shows that Conjecture~\ref{conj:part} holds in many cases; see Theorem~\ref{thm:parts}
below. Recall that for an integer $r \ge 2$ with prime factorization $2^{\alpha_0}\cdot p_1^{\alpha_1} \cdots p_t^{\alpha_t}$
the quantity $b(r)$ denotes $2^{\alpha_0}\cdot (p_1-1)^{\alpha_1} \cdots (p_t-1)^{\alpha_t}$.

\begin{theorem}
\label{thm:parts}
	Let $r \ge 2$ be an integer.
	Then Conjecture~\ref{conj:part} holds for $r$ and any partition $\P = \{P_1, \dots, P_\ell\}$ of $[n]$ 
	with~${|P_i| \le b(r)}$ for all~${i \in [\ell]}$.
\end{theorem}

\begin{proof}
	If $r$ is a power of two, this follows from Lemma~\ref{lem:2^t}.
	If $r$ is an odd prime this is guaranteed by Lemma~\ref{lem:bounds}.
	For general $r$ this follows by induction using Lemma~\ref{lem:prime-ind}.
\end{proof}

\begin{theorem}
\label{thm:main2}
	Let $r \ge 2$ be an integer.
	Let $k \ge 1$ and $n \ge 1$ be integers and $\s = (s_1, \dots, s_n)$ a vector of positive integers~${s_i \in [b(r)-1]}$.
	Then $\chi(\KG^r_{\s}(n,k)) \ge \lc \frac{(\sum_i s_i)-r(k-1)}{r-1} \rc$.
	In particular, $\chi(\KG^r_{s-1}(n,k)) \ge \lc \frac{(s-1)n-r(k-1)}{r-1} \rc$.
\end{theorem}

\begin{proof}
	Let $N= \sum_i s_i$ and let $\P = \{P_1, \dots, P_n\}$ be a partition of~$[N]$ with~${|P_i| = s_i}$.
	Then by Lemma~\ref{lem:transversal-red} $\chi(\KG^r_{\s}(n,k)) \ge \chi(\KG^r(N, k; \P))$.
	For $r$ a prime $\chi(\KG^r(N, k; \P)) =  \lc \frac{N-r(k-1)}{r-1} \rc$ by Lemma~\ref{lem:bounds}.
	For arbitrary $r$ we can induct on the number of prime divisors using Lemma~\ref{lem:prime-ind}.
\end{proof}

The lower bound $\chi(\KG^r_{s-1}(n,k)) \ge \lc \frac{(s-1)n-r(k-1)}{r-1} \rc$ (for parameters as in Theorem~\ref{thm:main2})
implies our main result, Theorem~\ref{thm:main}: Any coloring of the vertices of $\KG^r_{s-1}(n,k)$ with less
than $\lc \frac{(s-1)n-r(k-1)}{r-1} \rc$ colors results in a monochromatic hyperedge. The vertices of this hyperedge
are $r$ elements of $\binom{[n]}{k}$ (counted with multiplicities) that receive the same color and such that any 
$s$ of them (with multiplicities) have an empty intersection.

\begin{remark}
\label{rem}
Conjecture~\ref{conj:stab} implies Conjecture~\ref{conj:part} which in turn would 
yield the generalized Erd\H os--Kneser conjecture in full generality. Extending Lemma~\ref{lem:bounds}
to prime powers~$r$ would imply Conjecture~\ref{conj:part} for partitions with parts $P_i$ of size
at most~$r-1$. Looking back at the proof of~{\cite[Thm.~4.7]{frick2017-2}} this would follow from
extending the optimal colored Tverberg theorem of Blagojevi\'c, Matschke, and Ziegler~\cite{blagojevic2009}
from primes $r$ to prime powers.
\end{remark}

\section{A conjecture of Abyazi Sani and Alishahi}
\label{sec:conj}

Let $A \subset [n]$. Abyazi Sani and Alishahi~\cite{abyazi2017} consider the subhypergraph $\KG^r_A(n,k)$
of~$\KG^r(n,k)$ induced by the vertices that are $k$-subsets $\sigma \subset [n]$ with $\sigma \not\subset A$.
They conjecture:

\begin{conjecture}[Abyazi Sani and Alishahi~\cite{abyazi2017}] 
\label{conj:aa} 
	Let $n \ge 2rk$. Then 
	$$\chi(KG^r_A(n,k)) = \lc \frac{n - \max\{r(k-1),|A|\}}{r-1}\rc.$$
\end{conjecture}

Abyazi Sani and Alishahi show that the conjecture holds when $|A| \le 2(k-1)$ or $|A| \ge rk-1$, 
so the open case left is when $2k-1 \le |A| \le rk-2$. For this open case, they establish the inequality 
$$\lc \frac{n - r(k-1) - \lf \frac{|A|}{k}\rf }{r-1}\rc \le \chi(KG^r_A(n,k)) \le \lc \frac{n - \max\{r(k-1),|A|\}}{r-1}\rc.$$
Here we show:

\begin{theorem}
	Let $r \ge 2$.
	Then Conjecture~\ref{conj:aa} holds whenever~${|A| \le b(r)\cdot (k-1)}$. 
\end{theorem}

\begin{proof}
	Write $b$ for~$b(r)$. If $|A| \le b(k-1)$ we take a partition $P_1, \dots, P_\ell$ of $[n]$ such that $|P_i|= b$ for all $i \in [k-1]$, 
	$|P_i| \le b$ for all $k \le i \le \ell$, and $A \subset \bigcup_{i=1}^{k-1}P_i$. 
	Then $\KG^r(n,k;P_1, \dots, P_\ell)$ is a subgraph of $\KG^r_A(n,k)$. 
	Then the result follows from Theorem~\ref{thm:parts}.
\end{proof}

We will further extend this result to a larger class of Kneser hypergraphs in the next section.

\section{Lower bounds for the chromatic number of wide subhypergraphs}
\label{sec:subhyp}

When defining the hypergraph $\KG^r_A(n,k)$ we can by symmetry always assume that $A = \{1, 2, \dots, |A|\}$,
and thus we constrain $\KG^r(n,k)$ by only retaining those vertices that correspond to sets that are not contained 
in the initial interval of length $|A|$ in~$[n]$. More generally, we can disregard vertices that correspond to sets that
are contained in any interval of some fixed length, and still establish the same lower bounds for the chromatic number:
For an integer $t \ge 1$ denote by $\KG^r(n,k)_{t-\mathrm{wide}}$ the subhypergraph of $\KG^r(n,k)$ induced by
those vertices that correspond to $k$-subsets $\sigma$ in~$[n]$ that are not contained in any of the sets 
$\{i, i+1, \dots, i+t-1\}$ for~${i \in [n-t+1]}$. We call such sets $\sigma$ \emph{$t$-wide}.

Lower bounds for these subhypergraphs of Kneser graphs induced by ``wide'' $k$-sets in $[n]$ are interesting in the context
of the generalized Erd\H os--Kneser conjecture. For a partition $\P =\{P_1, \dots, P_\ell\}$ of $[n]$ and an integer $t \ge 1$
let $\KG^r(n,k;\P)_{t-\mathrm{wide}}$ be the subhypergraph of $\KG^r(n,k)$ whose vertices correspond to $k$-sets
$\sigma \subset [n]$ that are $t$-wide and transversal to $\P$, that is, $|\sigma \cap P_i| \le 1$ for each~${i \in [\ell]}$.
We will extend our lower bounds for the chromatic number to $\KG^r(n,k;\P)_{t-\mathrm{wide}}$ for certain values of $t$
and partitions with~${|P_i| \le r-1}$ provided that $r$ is a prime. Using the construction of Lemma~\ref{lem:transversal-red}
we thus obtain lower bounds for the chromatic number of certain natural subhypergraphs of~$\KG^r_{\s}(n,k)$.
This is analogous to a result of Schrijver~\cite{schrijver1978} who identifies vertex-critical subgraphs of 
Kneser graphs~$\KG^2(n,k)$. Natural subhypergraphs of $\KG^r(n,k)$ that still satisfy the same lower bound
for the chromatic number were identified by Meunier~\cite{meunier2011}, Alishahi and Hajiabolhassan~\cite{alishahi2015}, 
and the fourth author~\cite{frick2017-2}.

\begin{theorem}
\label{thm:t-wide}
	Let $k \ge 1$ be an integer, $r \ge 2$ a prime, and $n \ge rk$ an integer. Let $\P = \{P_1, \dots, P_\ell\}$ be a partition
	of~$[n]$ with~${|P_i| \le r-1}$. Let $t \le r(k-3)+2$. Then
	$$\chi(\KG^r(n,k;\P)_{t-\mathrm{wide}}) = \lc \frac{n-r(k-1)}{r-1} \rc.$$
\end{theorem}

\begin{proof}
	Denote the $k$-element subsets of $[n]$ that are $t$-wide by~$\binom{[n]}{k}_{t-\mathrm{wide}}$. Then according
	to Lemma~\ref{lem:tcd} we have to show that $\tcd^r(\binom{[n]}{k}_{t-\mathrm{wide}}) \ge n-r(k-1)$.
	Let $d$ be the integer satisfying $r(k-1)-1 \ge (r-1)d > r(k-2)$ and let $m = (r-1)d - r(k-2) - 1$.
	Let $K$ be the simplicial complex on vertex set $[n+m]$ whose missing faces are precisely the $t$-wide
	$k$-element subsets of~$[n]$. If we can show that there is a continuous map $f \colon K \longrightarrow \R^d$
	that does not identify $r$ points from $r$ pairwise disjoint faces of~$K$, then by definition of topological 
	$r$-colorabilty defect 
	$$\tcd^r\left(\binom{[n]}{k}_{t-\mathrm{wide}}\right) \ge n+m - (r-1)(d+1) = n + (r-1)d -r(k-2) - 1 - (r-1)(d+1) = n-r(k-1).$$
	
	We define the map $f \colon K \longrightarrow \R^d$ as an affine map. To define the images of the vertex set of $K$
	we need a result of Bukh, Loh, and Nivasch~\cite{bukh2016} on the existence of point configurations with certain 
	\emph{Tverberg partitions}, i.e., partitions of the point set into $r$ parts whose convex hulls all share a common point.
	In order to state their result we need the following definitions:
	A partition of $[(r-1)(d+1)+1]$ into $r$ parts $X_1, \dots, X_r$ is called \emph{colorful} if for each $k \in [d+1]$ 
	the set $Y_k = \{(r-1)(k-1)+1, \dots, (r-1)k+1\}$ satisfies $|Y_k \cap X_i| = 1$ for all~$i$. Given a sequence 
	$x_1, \dots, x_N$ of points in~$\R^d$ we say that a partition $X_1 \sqcup \dots \sqcup X_r$ of $[(r-1)(d+1)+1]$
	\emph{occurs} as a Tverberg partition if there is a subsequence $x_{i_1}, \dots, x_{i_{\widehat N}}$ of length
	$\widehat N = (r-1)(d+1)+1$ such that $\conv\{x_{i_k} \: | \: k \in X_1\} \cap \dots \cap \conv\{x_{i_k} \: | \: k \in X_r\} \ne \emptyset$.
	The result of Bukh, Loh, and Nivasch~\cite{bukh2016} can now be phrased as: There are arbitrarily long 
	sequences of points in $\R^d$ such that the Tverberg partitions that occur are precisely the colorful ones. 
	The significance of the number $\widehat N = (r-1)(d+1)+1$ is explained by Tverberg's theorem~\cite{tverberg1966},
	which guarantees that any point set in $\R^d$ of size $\widehat N$ can be split into $r$ parts whose
	convex hulls share a common point, while any point set of size $\widehat N-1$ and in sufficiently general position
	does not admit such a Tverberg partition. We remark that for the sequences whose Tverberg partitions are
	precisely the colorful ones, no subset of size $\widehat N-1$ can admit a Tverberg partition.
	
	Let $\{x_1, \dots, x_{n+m}\}$ be a point set in $\R^d$ whose Tverberg partitions are precisely the colorful ones.
	Define the map $f$ by sending vertex $i$ of $K$ to $x_i$ and extending linearly onto the faces of~$K$.
	Then the faces of~$K$ map precisely to the convex hulls of the images of its vertices.
	Now suppose that $K$ has pairwise disjoint faces $\sigma_1, \dots, \sigma_r$ such that
	$f(\sigma_1) \cap \dots \cap f(\sigma_r) \ne \emptyset$. Then without loss of generality (by Tverberg's theorem) the 
	$\sigma_i$ involve precisely $(r-1)(d+1)+1$ vertices. More precisely, choose a vertex $v$ of the convex
	polytope $f(\sigma_1) \cap \dots \cap f(\sigma_r)$, and for each $i \in [r]$ only consider a minimal set of 
	vertices of $\sigma_i$ whose convex hull captures~$v$. This reduces the total number of vertices to $(r-1)(d+1)+1$.
	There is a $j \in [r]$ such that $|\sigma_j \cap [n]| \ge k$.
	This is because $|\sigma_i \cap [n]| \le k-1$ for every $i \in [n]$ implies that the $\sigma_i$ involve at most
	$r(k-1)+m = (r-1)(d+1)$ vertices---a contradiction. The $\sigma_i$ are a colorful partition of some subset $A$
	of~$[n+m]$, which implies that $\sigma_j$ must have an element among the lowest $r$ elements of $A$ 
	as well as among the highest $r$ elements of~$A$. Then $\sigma_j$ cannot be contained in an interval
	of length~${(r-1)(d-1)-1}$. Now $(r-1)(d-1) \ge r(k-3)+2 \ge t$, and $\sigma_j$ contains a $k$-subset of 
	$[n]$ with two elements at distance at least~$t$. Thus $\sigma_j$ is a nonface of $K$, which is a contradiction
	since we chose $\sigma_1, \dots, \sigma_r$ to be faces of~$K$.
\end{proof}

For an integer $t \ge 1$ we denote by $\KG^r_{\s}(n,k)_{t-\mathrm{wide}}$ the subhypergraph of $\KG^r_{\s}(n,k)$
induced by those vertices that correspond to $t$-wide $k$-subsets of~$[n]$.

\begin{corollary}
\label{cor:subhyp}
	Let $r \ge 2$ be a prime. Let $k \ge 1$ and $n \ge k$ be integers and $\s = (s_1, \dots, s_n)$ a 
	vector of positive integers~${s_i \in [r-1]}$. Let $t \ge 1$ be the maximal integer such that
	$\sum_{i=j}^{j+t-1} s_i \le r(k-3)+2$ for every~${j \in [n-t+1]}$.
	Then $\chi(\KG^r_{\s}(n,k)_{t-\mathrm{wide}}) \ge \lc \frac{(\sum_i s_i)-r(k-1)}{r-1} \rc$.
	In particular, $\chi(\KG^r_{s-1}(n,k)_{t-\mathrm{wide}}) \ge \lc \frac{(s-1)n-r(k-1)}{r-1} \rc$,
	where $t = \lf \frac{r(k-3)+2}{s-1} \rf$.
\end{corollary}

\begin{proof}
	Let $N= \sum_i s_i$ and let $\P = \{P_1, \dots, P_n\}$ be the partition of~$[N]$ with 
	$$P_j = \left\{\sum_{i=1}^{j-1} s_i+1, \sum_{i=1}^{j-1} s_i+2,\dots, \sum_{i=1}^{j} s_i\right\}.$$
	Then as in the proof of Lemma~\ref{lem:transversal-red} 
	$\chi(\KG^r_{\s}(n,k)_{t-\mathrm{wide}}) \ge \chi(\KG^r(N, k; \P)_{\overline t-\mathrm{wide}})$,
	where~${\overline t = r(k-3)+2}$.
	This is because a $\overline t$-wide set in $[N]$ will project to a $t$-wide set in~$[n]$.
	Now use Theorem~\ref{thm:t-wide}.
\end{proof}

\section{Final remarks}

\begin{remark}
	Corollary~\ref{cor:subhyp} establishes the same lower bound for the chromatic number of
	$\KG^r_{\s}(n,k)_{t-\mathrm{wide}}$ for certain values of~$t$ as Theorem~\ref{thm:main2}
	for the hypergraph~$\KG^r_{\s}(n,k)$ of all $k$-subsets of~$[n]$. This can easily be extended
	to varying~$t$. For a vector $\s = (s_1, \dots, s_n)$ with $s_i \in [r-1]$ let $H$ be the 
	subhypergraph of $\KG^r_{\s}(n,k)$ where vertices correspond to those $k$-subsets
	$\sigma = \{\sigma_1, \dots, \sigma_k\} \subset [n]$ such that constructing a $k$-subset 
	of~$[N]$, $N =\sum_i s_i$, by choosing one point in each $P_{\sigma_i}$ as in the proof
	of Corollary~\ref{cor:subhyp} always results in a $(r(k-3)+2)$-wide set. Then
	$\chi(H) \ge \lc \frac{(\sum_i s_i)-r(k-1)}{r-1} \rc$ as before.
\end{remark}

\begin{remark}
	Our results extend to arbitrary set systems if $r$ is a prime. Let $\F$ be a system of
	subsets of~$[n]$. For a vector $\s = (s_1, \dots, s_n)$ with $s_i \in [r-1]$ denote by
	$\KG^r_{\s}(\F)$ the hypergraph on vertex set $\F$ where a multiset $\{\{F_1, \dots, F_r\}\}$,
	$F_i \in \F$, is a hyperedge if and only if $i$ is contained in at most $s_i$ of the sets
	$F_1, \dots, F_r$ for every~$i \in [n]$. Let $N = \sum_i s_i$ and let $\P =\{P_1, \dots, P_n\}$
	be a partition of $[N]$ with~$|P_i| = s_i$. Let $\F'$ be the system of subsets of~$[N]$ that
	for every $F\in \F$ contains the set $F' = \bigcup_{i \in F} P_i$. Then in complete analogy
	to Lemma~\ref{lem:transversal-red} one proves using Lemma~\ref{lem:tcd} that
	$$\chi(\KG^r_{\s}(\F)) \ge \chi(\KG^r(\F'; \P)) \ge \lc \frac{\tcd^r(\F')}{r-1} \rc.$$
\end{remark}

\section*{Acknowledgements}

This research was performed during the \emph{Summer Program for Undergraduate Research} 2017 at Cornell
University. The authors are grateful for the excellent research conditions provided by the program.
The authors would like to thank Maru Sarazola for many insightful conversations.
This manuscript was written while the fourth author was in residence at the Mathematical Sciences Research Institute 
in Berkeley, California, during the Fall 2017 semester supported by the National Science Foundation under 
Grant No. DMS--1440140.

\bibliographystyle{amsplain}


\end{document}